\documentclass[10pt]{amsart}
\usepackage{amsmath,amsfonts,euscript,amscd,amsthm,amssymb,upref,graphics}
\usepackage[all]{xy}
\theoremstyle{plain}
\newtheorem{Thm}{Theorem}[section]

\newtheorem{Lem}[Thm]{Lemma}
\newtheorem{Prop}[Thm]{Proposition}

\newtheorem*{Thm*}{Theorem}

\theoremstyle{definition}
\newtheorem{Def}[Thm]{Definition}

\newcommand{\Ceul}{\EuScript{C}}
\newcommand{\Oeul}{\EuScript{O}}

\theoremstyle{remark}
\newtheorem{Rem}[Thm]{Remark}

\newtheorem*{theorem*}{Theorem}
\newcommand{\C}{{\mathbb C}}
\begin{document}

\title[On the number of factors]%
{On the number of factors in the unipotent factorization of holomorphic mappings into $\text{SL}_2(\mathbb{C})$}
\author{Bj\"orn Ivarsson and Frank Kutzschebauch}
\address{Department of Natural Sciences, Engineering and Mathematics, Mid Sweden University, SE-851 70 Sundsvall, Sweden}
\email{Bjorn.Ivarsson@miun.se}

\address{Institute of Mathematics\\
University of Bern\\
Sidlerstrasse 5, CH--3012 Bern, Switzerland}
\email{frank.kutzschebauch@math.unibe.ch}
\date{December 5, 2010}

\setcounter{tocdepth}{2}
\begin{abstract}
We estimate the number of unipotent elements needed to factor a null-homotopic holomorphic map from a finite dimensional reduced Stein spaces $X$ into $\text{SL}_2(\mathbb{C})$ .
\end{abstract}
\thanks{Kutzschebauch supported by Schweizerische Nationalfonds grant 200021- 116165/1.}
\maketitle
\bibliographystyle{alpha}
\tableofcontents

\section{Introduction}
It's an old well known problem whether any matrix in ${SL}_n (R)$ for a ring $R$ (associative, commutative, with unit) is a product  of elementary (or equivalent unipotent) matrices with entries in the ring $R$. Especially interesting are the cases of  polynomial rings $\C [\C^m]$, of  rings  $\Ceul (X)$ of  
continuous functions on a topological space $X$ or of holomorphic functions $\Oeul  (X)$ on a Stein space. In the algebraic case for $n=2$ such a factorization does not always exist, the first counterexample was found by Cohn \cite{CohnSGL2R}. By Suslin's deep result for $n\ge3$ there is always a polynomial factorization \cite{SuslinSSLGRP}. There are no uniform bounds on the number of factors in the algebraic case \cite{vanderKallenSL3BWL}. 
The topological case  was solved by Vaserstein, there are uniform bounds on the number of matrices needed, but no  estimate for these numbers are known.
More precisely in \cite{VasersteinRMDPDFAO} Vaserstein proved the following result.

\begin{Thm*}
Let $X$ be a finite dimensional normal topological space and $f\colon X\to \mbox{SL}_n(\mathbb{C})$ be a null-homotopic continuous mapping. There exist a number $K$, depending only on the dimension of $X$ and $n$, and continuous mappings $g_1,\dots, g_{K}\colon X\to \mathbb{C}^{n(n-1)/2}$ such that $$f(x)=M_{1}(g_1(x))M_{2}(g_2(x))\dots M_{K}(g_{K}(x)).$$
\end{Thm*}
Here $M_j$ is defined as follows. For $j$ odd put $$M_j(g_j(x))= \left(\begin{matrix} 1 &   & 0 \\  & \ddots &  \\ g_j(x) & & 1 \end{matrix}\right)$$ and for $j$ even  $$M_j(g_j(x))= \left(\begin{matrix} 1 &   & g_j(x) \\  & \ddots &  \\ 0 & & 1 \end{matrix}\right).$$

The holomorphic case was solved by the authors \cite{IvarssonKutzschebauchHFMSLG} (announced in \cite {IvarssonKutzschebauchSGVP}) where the following  theorem is proven.
\begin{Thm*}
Let $X$ be a finite dimensional reduced Stein space and $f\colon X\to \mbox{SL}_n(\mathbb{C})$ be a null-homotopic holomorphic mapping. There exist a number $K$, depending only on the dimension of $X$ and $n$, and holomorphic mappings $g_1,\dots, g_{K}\colon X\to \mathbb{C}^{n(n-1)/2}$ such that $$f(x)=M_{1}(g_1(x))M_{2}(g_2(x))\dots M_{K}(g_{K}(x)).$$
\end{Thm*}

The proof of this theorem is done by reduction via the Oka-Grauert-Gromov-h-principle to the topological result of Vaserstein. This of course relates the number of factors needed for the holomorphic factorization to the numbers needed for the topological factorization. The relation given by the proof in \cite{IvarssonKutzschebauchHFMSLG} is probably not very sharp. To describe it let's introduce the following numbers:  Let $K(n,m,\Ceul)$ be the minimal number that all null-homotopic continuous mappings from normal topological spaces of dimension
$m$ into $\mbox{SL}_n(\mathbb{C})$ factorize as a product of $K(n,m,\Ceul)$ continuous unipotent matrices
(starting with a lower triangular one) and $K(n,m,\Oeul)$ be the minimal number that all null-homotopic holomorphic mappings from Stein spaces of dimension
$m$ into $\mbox{SL}_n(\mathbb{C})$ factorize as a product of $K(n,m,\Oeul)$ holomorphic unipotent matrices
(starting with a lower triangular one):   Then the relation given by our proof is:

$$K(n,m,\Oeul)\le 1 + \sum_{i= 2}^{n} (  K(i,2 m,\Ceul) +3 )$$ 

The proof goes by induction over the size of matrices since we could not prove that a certain fibration satisfies the Oka-Grauert-Gromov-h-principle. We had to project to the last row in order to construct the stratified spray. Adding $3$ at each step in the induction is used to avoid the singularity set of the fibration by the topological section provided from the Vaserstein theorem. 
Now of course a Stein space $X$ of dimension $m$ is a topological space of dimension $2m$, but Stein spaces are very special topological spaces (they have homology at most up to half of the real dimension) and holomorphic maps
are special among continuous maps. So if we introduce a number $K(n,m,\Ceul, \Oeul)$ to be the
minimal number $l$ such that all null-homotopic holomorphic mappings from Stein spaces of dimension
$m$ into $\mbox{SL}_n(\mathbb{C})$ factorize as a product of $l$ continuous unipotent matrices (starting with a lower triangular one), then the above mentioned proof gives
\begin{equation} K(n,m,\Oeul)\le 1 + \sum_{i= 2}^n (  K(i, m,\Ceul,\Oeul) +3 ) \label{1}
\end{equation}

which might be a better estimate since obviously $K(i,m,\Ceul,\Oeul)  \le K(i,2 m,\Ceul)$ and this inequality might be strict.

In general it is a very interesting question to find out bounds for the number of factors. Moreover
such bounds lead to concrete estimates for Kazhdan constants. Namely, in a small note \cite{IvarssonKutzschebauchKPTSLHF} the authors show that the groups $SL_n (\Oeul (X) )$ for a contractible Stein space $X$ admit Kazhdan's property (T) for $n\ge 3$. 

The present  paper is a starting point  of a  systematic study of the number of factors needed. We
have only results in the case $n=2$, i.e., matrices of size $2$ by $2$.

The first result of our paper  is an improvement of the estimate in equation (\ref{1}) by $2$ factors. We gain one factor compared to our earlier work by some easy trick, but the other factor is hard work. We find a stratified spray for a more complicated situation. Namely we can avoid the projection to the last row.
More precisely we prove:

\begin{Thm*} (see Theorem \ref{comparison})
Let $X$ be a finite dimensional Stein space and $f\colon X\to \mbox{SL}_2(\mathbb{C})$ be a holomorphic mapping that is null-homotopic.  Assume that there exists continuous mappings $g_1,\dots, g_{K}\colon X\to \mathbb{C}$ such that $$f(x)=M_{1}(g_1(x))M_{2}(g_2(x))\dots M_{K}(g_{K}(x)).$$
Then there exists holomorphic mappings $h_1,\dots, h_{K+2}\colon X\to \mathbb{C}$ such that $$f(x)=M_{1}(h_1(x))M_{2}(h_2(x))\dots M_{K+2}(h_{K+2}(x)).$$
\end{Thm*}

This result  (in our terms stated as  $K(2,m,\Oeul)\le 2 +  K(2, m,\Ceul,\Oeul) $) is almost sharp. In section \ref{ex} we work out in detail the Cohn example and find that
one needs to add one factor for the holomorphic factorization compared to the continuous 
factorization.

Our second result in the paper are some first exact estimates (see Theorems \ref{dim1} and \ref{dim2}):

$$K(2,1, \Oeul)=4 , \quad K(2,2,\Oeul) =5$$
 
 Clearly at least  $4$ factors are always needed since multiplication of $3$ elementary matrices is not surjective to $\mbox{SL}_2(\mathbb{C})$.
 
 We thank Shulim Kaliman and Anand Dessai for helpful conversations on topological matters.



\section{Proof of factorization for $\text{SL}_2(\mathbb{C})$}

\subsection{Overview of the proof}

We will give a new proof of the following theorem. 
\begin{Thm}
Let $X$ be a finite dimensional Stein space and $f\colon X\to \mbox{SL}_2(\mathbb{C})$ be a holomorphic mapping that is null-homotopic. Then there exist a number $K$, depending only on the dimension of $X$, and holomorphic mappings $g_1,\dots, g_{K}\colon X\to \mathbb{C}$ such that $$f(x)=M_{1}(g_1(x))M_{2}(g_2(x))\dots M_{K}(g_{K}(x)).$$
\end{Thm}

Sometimes below we will write $$U(g(x))= \left(\begin{matrix} 1 & g(x) \\ 0 & 1 \end{matrix}\right)$$ and $$L(g(x))= \left(\begin{matrix} 1 & 0 \\ g(x) & 1 \end{matrix}\right).$$

The strategy for proving this result is as follows. Below we will define a holomorphic mapping $\Phi_N\colon \mathbb{C}^N\to \mbox{SL}_2(\mathbb{C})$ which is surjective when $N\ge 4$. However it is submersive only outside a certain set $S_N$ so therefore we study $\Phi_N\colon \mathbb{C}^N\setminus S_N\to \mbox{SL}_2(\mathbb{C})$. This mapping will still be surjective so we have a surjective holomorphic submersion. If we can find a holomorphic $g\colon X \to \mathbb{C}^N\setminus S_N$ such that the diagram 

$$\xymatrix{ & \mathbb{C}^{N}\setminus S_N \ar[d]^{\Phi_{N}} \\ X \ar[r]_{f} \ar[ur]^{g} & \mbox{SL}_2(\mathbb{C})}$$
is commutative we will have found the desired factorization. To find this mapping we will pull-back the bundle $\xi=(\mathbb{C}^N\setminus S_N, \Phi_N,\mbox{SL}_2(\mathbb{C}))$ with $f$ to get the bundle $f^*\xi=(f^*(\mathbb{C}^{N}\setminus S_N),f^*\Phi_{N},X)$ via the commutative diagram 

$$\xymatrix{f^*(\mathbb{C}^{N}\setminus S_N) \ar[r]^{f_\xi} \ar[d]_{f^*\Phi_{N}} & \mathbb{C}^{N}\setminus S_N \ar[d]^{\Phi_{N}} \\ X \ar[r]_{f} & \mbox{SL}_2(\mathbb{C})}$$
and a section of this bundle will correspond to a factorization into a product of unipotent matrices. The result by Vaserstein gives us a continuous mapping that after some manipulation can be made to avoid $S_N$. After this change of the mapping it pulls back to a continuous section of the pull-back bundle. This manipulation is the geometric reason for the increase in the number of factors needed in the holomorphic case. We will construct complete holomorphic vector fields on the fibers of $\Phi_N$ and consequently on the fibers of $f^*\Phi_N$. We can then use results of Gromov \cite{GromovOPHSEB} and Forstneri\v c \cite{ForstnericOPSSFB} to conclude that the continuous section can be homotopically deformed to a holomorphic section. We then have proven that the desired holomorphic factorization exists.

  
\subsection{The mapping $\Phi_N$ and it's fibers} 
Define the mapping $\Phi_N\colon \mathbb{C}^N \to \mbox{SL}_2(\mathbb{C})$ as 
\begin{equation*}\Phi_N(z_1,\dots,z_N)=M_{1}(z_1) M_{2}(z_2)\cdots M_{N}(z_N).\end{equation*} 
Let us investigate where the mapping $\Phi_N $ is submersive. 
\begin{Lem}\label{L:submersive}
The mapping $\Phi_N$ is submersive exactly at points $$z\in \mathbb{C}^N\smallsetminus \{(z_1,0,\dots,0,z_N)\}$$ when $N\ge 4$. 
\end{Lem}
\begin{proof} 
We begin by studying when the differential of $\Phi_3$ spans a 3-dimensional space. Therefore we study the equation \[ \lambda_1\frac{\partial \Phi_3}{\partial z_1}+\lambda_2\frac{\partial \Phi_3}{\partial z_2}+\lambda_3\frac{\partial \Phi_3}{\partial z_3}=0 \] which we write as, using \[ e_{12}=\left(\begin{matrix} 0 & 1 \\ 0 & 0 \end{matrix}\right) \mbox{ and } e_{21}=\left(\begin{matrix} 0 & 0 \\ 1 & 0 \end{matrix}\right), \] \[\lambda_1e_{21}U(z_2)L(z_3)+\lambda_2L(z_1)e_{12}L(z_3)+\lambda_3L(z_1)U(z_2)e_{21}=0. \] We now multiply this equation from the right by \[ \left(L(z_1)U(z_2)L(z_3)\right)^{-1}=L(-z_3)U(-z_2)L(-z_1).\] This doesn't change the dimension of the span of the differential and since $e_{21}L(-z_3)=e_{21}$ we get the equation \[ \begin{aligned}\lambda_1 e_{21}L(-z_1)&+\lambda_2L(z_1)e_{12}U(-z_2)L(-z_1)+\lambda_3L(z_1)U(z_2)e_{21}L(-z_3)U(-z_2)L(-z_1)=\\ &= \lambda_1 e_{21}+\lambda_2L(z_1)e_{12}L(-z_1)+\lambda_3L(z_1)U(z_2)e_{21}U(-z_2)L(-z_1)=0\end{aligned}\] Notice that we have an equation that is independent of $z_3$. We can therefore put $z_3=0$ in the original equation. If we now multiply the original equation in the same way but from the left we see in the same way that we can put $z_1=0$. Our simplified equation takes the form \[\lambda_1 e_{21}U(z_2)+\lambda_2 e_{12}+\lambda_3 U(z_2)e_{21}=0. \] Once again we multiply the equation from the right with an inverse, this time $U(-z_2)$. We do this in order to use the basis, $e_{12}, e_{21},$ and \[d_{12}=e_{11}-e_{22}=\left(\begin{matrix} 1 & 0 \\ 0 & -1 \end{matrix}\right)\] for the Lie algebra $\mathfrak{sl}_2(\mathbb{C})$ which is the tangent space for $\mbox{SL}_2(\mathbb{C})$ at the identity. Doing so we get \[\begin{aligned}\lambda_1& e_{21}+\lambda_2 e_{12}+\lambda_3 U(z_2)e_{21}U(-z_2)=\\&=\lambda_1e_{21}+\lambda_2e_{12}+\lambda_3(e_{21}-z_2e_{22}+z_2e_{11}-z_2^2e_{12})=\\ &=(\lambda_1+\lambda_3)e_{21}+(\lambda_2-z_2^2\lambda_3)e_{12}+z_2\lambda_3d_{12}=0.\end{aligned} \] Therefore the span is 3-dimensional when \[\det \left(\begin{matrix} 1 & 0 & 1 \\ 0 & 1 & -z_2^2 \\ 0 & 0 & z_2\end{matrix} \right)=z_2\neq 0.\] A similar calulation shows that $\Psi_3(z_2,z_3,z_4)=U(z_2)L(z_3)U(z_4)$ is submersive when $z_3\neq 0$. Now assume that $z_i\neq 0$ for some $2\le i \le N-1$. We can then find only trivial solutions to \[\lambda_1 \frac{\partial \Phi_N}{\partial z_{i-1}}+\lambda_2 \frac{\partial \Phi_N}{\partial z_i}+\lambda_3\frac{\partial \Phi_N}{\partial z_{i+1}}=A\left(\lambda_1 \frac{\partial \Phi_3}{\partial z_{i-1}}+\lambda_2 \frac{\partial \Phi_3}{\partial z_i}+\lambda_3\frac{\partial \Phi_3}{\partial z_{i+1}} \right)B=0, \] for appropriate matrices $A$ and $B$, by what we already have proved. Therefore the result follows.
\end{proof}

\subsection{The fibers of $\Phi_N$}

In order to understand the fibers of $\Phi_N$ let us do the following calculations. We have $$\left(\begin{matrix} 1 & 0 \\ z_1 & 1 \end{matrix}\right)\left(\begin{matrix} Q_1 & Q_2 \\ Q_3 & Q_4 \end{matrix}\right)\left(\begin{matrix} 1 & z_{2n} \\ 0 & 1 \end{matrix}\right)=\left(\begin{matrix} Q_1 & Q_2+Q_1z_{2n} \\ Q_3+Q_1z_{1} & Q_4+Q_2z_2+Q_3z_{2n}+Q_1z_1z_{2n} \end{matrix}\right) $$ and $$\left(\begin{matrix} 1 & 0 \\ z_1 & 1 \end{matrix}\right)\left(\begin{matrix} Q_1 & Q_2 \\ Q_3 & Q_4 \end{matrix}\right)\left(\begin{matrix} 1 & 0 \\ z_{2n+1} & 1 \end{matrix}\right)=\left(\begin{matrix} Q_1+Q_2z_{2n+1} & Q_2 \\ Q_3+Q_1z_1+Q_4z_{2n+1}+Q_2z_1z_{2n+1} & Q_4+Q_2z_{1} \end{matrix}\right). $$ Here $Q_1$, $Q_2$, $Q_3$, and $Q_4$ are polynomials in $z_2,\dots, n_{2n-1}$ or $z_2,\dots,z_{2n}$ depending on $N$ being even or odd. Remember that the map $\Phi_N$ is non-submersive precisely when all these variables are 0. That is at points where $$\left(\begin{matrix} Q_1 & Q_2 \\ Q_3 & Q_4 \end{matrix}\right)=\left(\begin{matrix} 1 & 0 \\ 0 & 1 \end{matrix}\right) .$$

\subsubsection{The fibers when $N$ is even}
We now try to understand the fibers for the map $\Phi_{2n}\colon \mathbb{C}^{2n}\setminus S_{2n}\to \mbox{SL}_2(\mathbb{C})$. Here we have the equations 
\begin{enumerate}
\item{$Q_1=a$} 
\item{$Q_2+Q_1z_{2n}=b$}
\item{$Q_3+Q_1z_1=c$} 
\item{$Q_4+Q_2z_1+Q_3z_{2n}+Q_1z_1z_{2n}=a$}
\end{enumerate} 
that describes the fiber $$\Phi_{2n}^{-1}\left(\left(\begin{matrix} a & b \\ c & d \end{matrix}\right)\right).$$ What will be important is that the fibers can be described as graphs over smooth manifolds. Note that we have $ad-bc=1$ and $Q_1Q_4-Q_2Q_3=1$.

We will first study the generic case when $a\neq 0$ and show that in this case equations (1), (2), and (3) implies (4). We have 
\begin{align*}
&a(Q_4+Q_2z_1+Q_3z_{2n}+Q_1z_1z_{2n}-d)= aQ_4+aQ_3z_1+aQ_2z_{2n}+aQ_4z_1z_{2n}-ad =\\
& = Q_1Q_4-ad+Q_1Q_3z_1+Q_1Q_2z_{2n}+Q_1^2z_1z_{2n}= \\ &= Q_2Q_3-bc+Q_1Q_3z_1+Q_1Q_2z_{2n}+Q_1^2z_1z_{2n}=\\
& = (b-Q_1z_{2n})(c-Q_1z_{1})-bc+(c-Q_4z_{2n})Q_1z_1+(b-Q_1z_{2n})Q_1z_{2n}+Q_1^2z_1z_{2n}=0
\end{align*}
and since $a\neq 0$ we see that (4) is automatically fulfilled if (1), (2), and (3) are.


We show that $Q_1=a$ defines a smooth surface when $a\neq 1$. We claim that the singularity on $Q_1=1$ is located where all variables are zero. We begin with case $N=4$. Here $Q_1(z_2,z_3)=1+z_2z_3$. We immediately see that $dQ_1=z_3\ dz_2+ z_2\ dz_3=0$ precisely when $z_2=z_3=0$ and this is what we want to prove in this case. Now assume that the claim is true when $N=2n-2$ and we study $N=2n$. We have 
\begin{equation*}
\begin{aligned}
&\left(\begin{matrix} Q_1(z_2,\dots,z_{2n-1}) & Q_2(z_2,\dots,z_{2n-1}) \\ Q_3(z_2,\dots,z_{2n-1}) & Q_4(z_2,\dots,z_{2n-1}) \end{matrix}\right)= \\ &=\left(\begin{matrix} \widetilde{Q}_1(z_2,\dots,z_{2n-3}) & \widetilde{Q}_2(z_2,\dots,z_{2n-3}) \\ \widetilde{Q}_3(z_2,\dots,z_{2n-3}) & \widetilde{Q}_4(z_2,\dots,z_{2n-3}) \end{matrix}\right) 
\left(\begin{matrix} 1 & z_{2n-2} \\ 0 & 1 \end{matrix}\right)
\left(\begin{matrix} 1 & 0 \\ z_{2n-1} & 1 \end{matrix}\right)
\end{aligned}
\end{equation*}
and we are interested in
\begin{equation*}
Q_1=z_{2n-1}\widetilde{Q}_2+(1+z_{2n-2}z_{2n-1})\widetilde{Q}_1.
\end{equation*}
We want to show that $dQ_1=0$ precisely where all variables are zero. We get 
\begin{equation*}
dQ_1=z_{2n-1}\ d\widetilde{Q}_2+\widetilde{Q}_2\ dz_{2n-1}+(1+z_{2n-2}z_{2n-1})\ d\widetilde{Q}_1+\widetilde{Q}_1(z_{2n-2}\ dz_{2n-1}+z_{2n-1}\ dz_{2n-2})
\end{equation*}
and hence $dQ_1=0$ if and only if 
\begin{itemize}
\item{$z_{2n-1}\widetilde{Q}_1=0$} 
\item{$\widetilde{Q}_2+z_{2n-2}\widetilde{Q}_1=0$}
\item{$z_{2n-1}\ d\widetilde{Q}_2+(1+z_{2n-2}z_{2n-1})\ d\widetilde{Q}_1=0$}
\end{itemize} 
If $\widetilde{Q}_1=0$ then $\widetilde{Q}_2=0$ and this implies that $Q_1=a=0$ and we are not considering this case here. Hence we must have $z_{2n-1}=0$ and this implies that $d\widetilde{Q}_1=0$. The induction hypothesis implies that $z_2=\dots=z_{2n-3}=0$. This implies in turn that $\widetilde{Q}_1=1$ and $\widetilde{Q}_2=0$ which implies that $z_{2n-2}=0$ and the claim follows by induction. 

It is now easy to write the generic fibers as graphs over $Q_1=a\neq 0$. We immediately see from (2) and (3) that $z_1=(c-Q_3)/a$ and $z_{2n}=(b-Q_2)/a$.

In the non-generic case $a=0$ we see that (1) and (2) implies (3) since $bc=Q_2Q_3=-1$ in this case. We see that $Q_1=0$ and $Q_2=b$ implies $Q_3=c$ since $b(Q_3-c)=Q_2Q_3-bc=0$. So in this case we need to investigate
\begin{itemize}
\item{$Q_1=0$} 
\item{$Q_2=b$}
\item{$Q_4+cz_1+bz_{2n}=d$}
\end{itemize}
Since we have a mapping into the special linear group $b\neq 0$. We claim that $Q_2=b$ defines a smooth complex hypersurface of in $\mathbb{C}^{2n}\cap \{z\in \mathbb{C}^{2n}:z_{2n-1}=z_{2n}=0\}$ and that we can write the non-generic fibers as graphs over this hypersurface. To see this notice that \begin{equation*}
\begin{aligned}
&\left(\begin{matrix} Q_1(z_2,\dots,z_{2n-1}) & Q_2(z_2,\dots,z_{2n-1}) \\ Q_3(z_2,\dots,z_{2n-1}) & Q_4(z_2,\dots,z_{2n-1}) \end{matrix}\right)=\\&=\left(\begin{matrix} R_1(z_2,\dots,z_{2n-2}) & R_2(z_2,\dots,z_{2n-2}) \\ R_3(z_2,\dots,z_{2n-2}) & R_4(z_2,\dots,z_{2n-2}) \end{matrix}\right) 
\left(\begin{matrix} 1 & 0 \\ z_{2n-1} & 1 \end{matrix}\right)
\end{aligned}
\end{equation*}
and we get the equations $R_2=b$ and $R_1+z_{2n-1}R_2=0$. We see that we can express $z_{2n-1}=-R_1/b$ and then $z_{2n}=(d-Q_4-cz_1)/b$.

We have 
\begin{itemize}
\item{$Q_2=\widetilde{Q}_2+z_{2n-2}\widetilde{Q}_1$}
\item{$Q_1=z_{2n-1}\widetilde{Q}_2+(1+z_{2n-2}z_{2n-1})\widetilde{Q}_1=\widetilde{Q}_1+z_{2n-1}Q_2$}
\end{itemize}
and we want to show that $Q_2=b\neq0$ defines a smooth complex hypersurface. Hence we want to show that $dQ_2\neq 0$ when $Q_1=0$ and $Q_2=b$. In order to do that we begin by showing that $dQ_1\wedge dQ_2\neq 0$. We see that 
\begin{equation*}
dQ_1\wedge dQ_2=(d\widetilde{Q}_1+z_{2n-1}\ dQ_2+ Q_2 \ dz_{2n-1})\wedge dQ_2=d\widetilde{Q}_1\wedge dQ_2 + Q_2 \ dz_{2n-1}\wedge dQ_2
\end{equation*}
and therefore $Q_2=0$ when $dQ_1\wedge dQ_2=0$. But $Q_2=b\neq 0$ so therefore we have $dQ_1\wedge dQ_2\neq 0$ when $Q_1=0$ and $Q_2=b\neq 0$. When $N=4$ we have that $Q_2=z_2$ so here we obviously have $dQ_2\neq 0$. Now assume that $d\widetilde{Q}_2\neq 0$ and study $$dQ_2=d(\widetilde{Q}_2+z_{2n-2}\widetilde{Q}_1)=d\widetilde{Q}_2+z_{2n-2}\ d\widetilde{Q}_1+\widetilde{Q}_1\ dz_{2n-2}.$$ We see that $dQ_2=0$ implies that $\widetilde{Q}_1=0$. If $dQ_2=0$ then $0=d\widetilde{Q}_2\wedge dQ_2=z_{2n-2}\ d\widetilde{Q}_2\wedge d\widetilde{Q}_1$ implies $z_{2n-2}=0$. But this in turn implies that $dQ_2=d\widetilde{Q}_2=0$ which is a contradiction. So by induction $dQ_2\neq 0$ and we see that also in the non-generic case the fibers are graphs over smooth manifolds.

We have shown that the fibers of $\Phi_{2n}$ can be described as graphs over smooth manifolds. This will let us construct $\mathbb{C}$-complete holomorphic vector fields that in turn will give us the fiber spray that we need.

\subsubsection{The fibers when $N$ is odd}
Using the same reasoning we can describe the fibers of the map $\Phi_{2n+1}\colon \mathbb{C}^{2n+1}\setminus S_{2n+1}\to \mbox{SL}_2(\mathbb{C})$ as graphs over smooth manifolds. The situation is very similar to when $N$ is even only here in the generic case we write the fiber as a graph over $Q_2=b\neq 0$ and in the non-generic case as a graph over $Q_1=a\neq 0$. We will skip doing the details.

\subsection{Stratified sprays associated with $\Phi_N$}

We will introduce the concept of a spray associated with a holomorphic submersion following \cite{GromovOPHSEB} and \cite{ForstnericOPHSWS}. First we introduce some notation and terminology. Let $h\colon Z \to X$ be a holomorphic submersion of a complex manifold $Z$ onto a complex manifold $X$. For any $x\in X$ the fiber over $x$ of this submersion will be denoted by $Z_x$. At each point $z\in Z$ the tangent space $T_zZ$ contains {\it the vertical tangent space} $VT_zZ=\ker Dh$. For holomorphic vector bundles $p\colon E \to Z$ we denote the zero element in the fiber $E_z$ by $0_z$.

\begin{Def}
Let $h\colon Z \to X$ be a holomorphic submersion of a complex manifold $Z$ onto a complex manifold $X$. A spray on $Z$ associated with $h$ is a triple $(E,p,s)$, where $p\colon E\to Z$ is a holomorphic vector bundle and $s\colon E\to Z$ is a holomorphic map such that for each $z\in Z$ we have 
\begin{itemize}
\item[(i)]{$s(E_z)\subset Z_{h(z)}$,}
\item[(ii)]{$s(0_z)=z$, and}
\item[(iii)]{the derivative $Ds(0_z)\colon T_{0_z}E\to T_zZ$ maps the subspace $E_z\subset T_{0_z}E$ surjectively onto the vertical tangent space $VT_zZ$.}
\end{itemize} 
\end{Def} 

\begin{Rem}
We will also say that the submersion admits a spray.
\end{Rem}

One way of constructing sprays associated with a holomorphic submersion is to find finitely many $\mathbb{C}$-complete vector fields that are tangent to the fibers and span the tangent space of the fibres at all points in $Z$. One can then use the flows $\varphi_j^t$ of these vector fields $V_j$ to define $s\colon Z\times \mathbb{C}^N\to Z$ via $s(z,t_1,\dots, t_N)=\varphi_1^{t_1}\circ \dots \circ \varphi_N^{t_N}(z)$ which gives a spray associated with $h$.

\subsubsection{The even-dimensional case}
The case when $N$ is even is only superficially different from the case when $N$ is odd. We do the even-dimensional case carefully and leave out most of the details for the odd-dimensional case. We begin with the generic case and study the polynomial equation $\Phi_N^{11}=Q_1=a\neq 0$. For ease of notation put $P=\Phi_N^{11}$. Let $P_j=\partial P/\partial z_j$ and define the complete vector fields $$V_{kl}=P_l\frac{\partial}{\partial z_k}-P_k\frac{\partial}{\partial z_l}.$$ We immediately see that $V_{kl}(P-a)=0$. The vector fields $V_{kl}$, $2\le k < l \le N-1$, spans the tangent spaces of the fibers at points where $P=a$ defines a manifold. This is because $dP=\partial P\neq 0$ at these points. Since we already know that $P=a\neq 0$ defines a manifold when $a\neq 1$ and the points where $P=1$ has singularities are located in $S_N$ we get the spanning property for the vector fields. The vector fields are complete since the coefficient functions are no more than linear in each variable, $P_l\frac{\partial}{\partial z_k}$ is independent of $z_l$, and $P_k \frac{\partial}{\partial z_l}$ is independent of $z_k$. Finally lift the vector fields onto the fibers of the holomorphic submersion and we have handled the generic case. We need to construct new vector fields to handle the case non-generic case $a=0$ and in order to use the results of \cite{ForstnericEHSCS} we will need so called stratified sprays. 

\begin{Def}
We say that a submersion $h\colon Z\to X$, where $X$ is a Stein space, admits stratified sprays if there is a descending chain of closed complex subspaces $X=X_m\supset \cdots \supset X_0$ such that each stratum $Y_k = X_k\setminus X_{k-1}$ is regular and the restricted submersion $h\colon Z|_{Y_k}\to Y_k$ admits a spray over a small neighborhood of any point $x\in Y_k$.  
\end{Def}

In \cite{ForstnericOPSSFB}, see also \cite{ForstnericEHSCS}, the following theorem is proven.
\begin{Thm}\label{t:forstnericprezelj}
Let $X$ be a Stein space with a descending chain of closed complex subspaces $X=X_m\supset \cdots \supset X_0$ such that each stratum $Y_k = X_k\setminus X_{k-1}$ is regular. Assume that $h\colon Z \to X$ is a holomorphic submersion which admits stratified sprays then any continuous section $f_0\colon X\to Z$ such that $f_0|_{X_0}$ is holomorphic can be deformed to a holomorphic section $f_1\colon X\to Z$ by a homotopy that is fixed on $X_0$. 
\end{Thm}

Our stratification will be $\mbox{SL}_2(\mathbb{C})\supset X_1\supset \emptyset$ where $$X_1=\left\{ \left(\begin{matrix} a & b \\ c & d\end{matrix}\right)\in \mbox{SL}_2(\mathbb{C}) ; a=0\right\}.$$ We have constructed a spray associated with $\Phi_N\colon \mathbb{C}^N|_{\mbox{SL}_2(\mathbb{C})\setminus X_1}\to \mbox{SL}_2(\mathbb{C})\setminus X_1$ and need one associated with $\Phi_N\colon \mathbb{C}^N|_{X_1}\to X_1$. So we need to consider the case $a=0$. The construction of complete vector fields will be done in the same way as the case $a\neq 0$ but with some minor modifications. Let $P=\Phi_N^{12}(z_1,z_2,\dots,z_{2n-2},0,0)$ (notice that we have $P=Q_2$ when $Q_1=0$) and define $$W_{kl}=P_l\frac{\partial}{\partial z_k}-P_k\frac{\partial}{\partial z_l}$$ for $1\le k< l \le N-2$ where $P_j=\partial P/\partial z_j$. These vector fields spans the tangent space of $Q_2=b\neq 0$ and are integrable for the same reason as in the case $a\neq 0$. Once again we lift the vector fields onto the fiber and we are done.

\subsubsection{The odd-dimensional case}
This works in the same way as the even-dimensional case. We only note that the stratification will be $\mbox{SL}_2(\mathbb{C})\supset X_1\supset \emptyset$ where $$X_1=\left\{ \left(\begin{matrix} a & b \\ c & d\end{matrix}\right)\in \mbox{SL}_2(\mathbb{C}) ; b=0\right\}$$ and that the bad set $S_N$ is contained in the fibers over $X_1$. However these points can be avoided and hence removed from the fibration and therefore presents no problem for us.

\section{Unipotent generation of null-homotopic holomorphic mappings into $\text{SL}_2( \mathbb{C})$}

\subsection{The first result on the number of factors}
 
Consider $\Phi_{N}\colon \mathbb{C}^{N}\to \mbox{SL}_2(\mathbb{C})$. By Lemma \ref{L:submersive} we know that the mapping is submersive outside the set $$S_N=\{z\in \mathbb{C}^{N}; z=(z_1,0,\dots,0,z_{N})\}.$$ Therefore the bundle $\xi=(\mathbb{C}^{N}\setminus S_N,\Phi_{N},\mbox{SL}_2(\mathbb{C}))$ has a submersive projection. Here we abuse notation slightly and write $\Phi_{N}=\Phi_{N}|_{\mathbb{C}^{N}\setminus S_N}$. The pull-back bundle $$f^*\xi=(f^*(\mathbb{C}^{N}\setminus S_N),f^*\Phi_{N},X)$$ also has a submersive projection. Here the total space of $f^*\xi$ is the subspace $$f^*(\mathbb{C}^{N}\setminus S_N)=\{(x,z)\in X\times (\mathbb{C}^{N}\setminus S_N);f(x)=\Phi_{N}(z)\}$$ and the projection is $f^*\Phi_{N}(x,z)=x$. We also have $f_\xi\colon f^*(\mathbb{C}^{N}\setminus S_N)\to \mathbb{C}^{N}\setminus S_N$ defined as $f_\xi(x,z)=z$. We get the commutative diagram $$\xymatrix{f^*(\mathbb{C}^{N}\setminus S_N) \ar[r]^{f_\xi} \ar[d]_{f^*\Phi_{N}} & \mathbb{C}^{N}\setminus S_N \ar[d]^{\Phi_{N}} \\ X \ar[r]_{f} & \mbox{SL}_2(\mathbb{C})}$$ and this induces a commutative diagram for the tangent spaces which lets us conclude that $f^*\xi$ has submersive projection and we saw in the previous section that it admits a stratified spray.

By Vaserstein's result there exists a continuous mapping $g\colon X\to \mathbb{C}^{N}$ such that $f(x)=\Phi_{N}(g(x))$. Assume that we know that $g(X)\cap S_N=\emptyset$. Then we get a global continuous section $f^*g\colon X \to f^*(\mathbb{C}^{N}\setminus S_N)$ defined as $f^*g(x)=(x,g(x))$ and we can use Theorem \ref{t:forstnericprezelj} to deform this section into a holomorphic section and this will show that we can write the map $f$ as a product of elementary matrices with holomorphic entries. In general we don't know if the continuous mapping $g$ is such that $g(X)\cap S_N=\emptyset$ but we can add two matrices in the factorization to make sure that we avoid the bad set. Assume that 
\begin{equation*}
f(x)= \left(\begin{matrix} 1& g_1(x) \\ 0 & 1 \end{matrix}\right) \left(\begin{matrix} 1 & 0 \\ g_2(x) & 1 \end{matrix}\right)
 \dots \left(\begin{matrix} 1 & 0 \\ g_N(x) & 1 \end{matrix}\right).
\end{equation*}
Trivially we have 
\begin{equation*}
\begin{aligned}
f(x)&= \left(\begin{matrix} 1& g_1(x) \\ 0 & 1 \end{matrix}\right) \left(\begin{matrix} 1& 1 \\ 0 & 1 \end{matrix}\right)\left(\begin{matrix} 1& 0 \\ 0 & 1 \end{matrix}\right)\left(\begin{matrix} 1& -1 \\ 0 & 1 \end{matrix}\right) \left(\begin{matrix} 1 & 0 \\ g_2(x) & 1 \end{matrix}\right)
 \dots \left(\begin{matrix} 1 & 0 \\ g_N(x) & 1 \end{matrix}\right) = \\ &= \left(\begin{matrix} 1& g_1(x)+1 \\ 0 & 1 \end{matrix}\right) \left(\begin{matrix} 1& 0 \\ 0 & 1 \end{matrix}\right)\left(\begin{matrix} 1& -1 \\ 0 & 1 \end{matrix}\right) \left(\begin{matrix} 1 & 0 \\ g_2(x) & 1 \end{matrix}\right)
 \dots \left(\begin{matrix} 1 & 0 \\ g_N(x) & 1 \end{matrix}\right)
\end{aligned}
\end{equation*}
and we see that we have a new factorization corresponding to the map $$\widetilde{g}(x)=(g_1(x)+1,0,-1,g_2(x),\dots,g_N(x))\in \mathbb{C}^{N+2}$$ which avoids the bad set $S_{N+2}$. The same trick also works when $N$ is odd. Therefore we have
\begin{Thm} \label{comparison}
Let $X$ be a finite dimensional Stein space and $f\colon X\to \mbox{SL}_2(\mathbb{C})$ be a holomorphic mapping that is null-homotopic.  Assume that there exists continuous mappings $g_1,\dots, g_{K}\colon X\to \mathbb{C}$ such that $$f(x)=M_{1}(g_1(x))M_{2}(g_2(x))\dots M_{K}(g_{K}(x)).$$
Then there exists holomorphic mappings $h_1,\dots, h_{K+2}\colon X\to \mathbb{C}$ such that $$f(x)=M_{1}(h_1(x))M_{2}(h_2(x))\dots M_{K+2}(h_{K+2}(x)).$$
\end{Thm}

\section{The example}\label{ex}
The counterexample to factorization  in the algebraic case ($\mathbb
 {C} [z, w]$) of Cohn is $$\left(\begin{matrix} 1+zw & z^2 \\ -w^2 & 1-zw  \end{matrix}\right) $$ and we will find a holomorphic and a topological factorization of this matrix. The minimal number of factors in the continuous case will be 4 and we will show that in the holomorphic case we need 5 factors. 

Let us start with giving a concrete holomorphic factorization. 
Of course the existence of a factorization with $5$ factors follows also from Theorem \ref{dim2}.
The first step is $$\left(\begin{matrix} 1 & -h_1(z,w) \\ 0 & 1 \end{matrix}\right)\left(\begin{matrix} 1+zw & z^2 \\ -w^2 & 1-zw \end{matrix}\right)= \left(\begin{matrix} e^{zw} & z^2 - (1-zw)h_1(z,w) \\ -w^2 & 1-zw \end{matrix}\right) $$ where $$h_1(z,w)=\frac{e^{zw}-1-zw}{w^2}.$$ Putting $h_2(z,w)=(1-w^2)e^{-zw}$ we get $$ \left(\begin{matrix} 1 & 0 \\ -h_2(z,w) & 1 \end{matrix}\right) \left(\begin{matrix} e^{zw} & z^2 - (1-zw)h_1(z,w) \\ -w^2 & 1-zw   \end{matrix}\right) = \left(\begin{matrix} e^{zw} & H(z,w) \\ 1 & G(z,w) \end{matrix}\right).$$ Using $h_3(z,w)=e^{zw}-1$ and $h_4(z,w)=1$ we see $$ \left(\begin{matrix} 1 & -h_3(z,w) \\ 0 & 1 \end{matrix}\right) \left(\begin{matrix} e^{zw} & H(z,w) \\ 1 & G(z,w)\end{matrix}\right) = \left(\begin{matrix} 1 & H_2(z,w) \\ 1 & G_2(z,w) \end{matrix}\right)$$ and  $$ \left(\begin{matrix} 1 & 0 \\ -h_4(z,w) & 1 \end{matrix}\right) \left(\begin{matrix} 1 & H_2(z,w) \\ 1 & G_2(z,w)\end{matrix}\right) = \left(\begin{matrix} 1 & H_2(z,w) \\ 0 & G_3(z,w) \end{matrix}\right).$$ Now $G_3(z,w)=1$ so we have a factorization. We get  $$ \left(\begin{matrix} 1+zw & z^2 \\ -w^2 & 1-zw  \end{matrix}\right) = \left(\begin{matrix} 1 & h_1 \\ 0 & 1 \end{matrix}\right) \left(\begin{matrix} 1 & 0 \\ h_2 & 1 \end{matrix}\right) \left(\begin{matrix} 1 & h_3 \\ 0 & 1 \end{matrix}\right) \left(\begin{matrix} 1 & 0 \\ h_4 & 1 \end{matrix}\right) \left(\begin{matrix} 1 & H_2 \\ 0 &1 \end{matrix}\right).$$ 

Now we analyze what it means  to find a factorization using just 4 matrices. If we can find $h_1, h_2, h_3,$ and $h_4$ such that $$ \left(\begin{matrix} 1+zw & z^2 \\ -w^2 & 1-zw  \end{matrix}\right) = \left(\begin{matrix} 1 & h_1 \\ 0 & 1 \end{matrix}\right) \left(\begin{matrix} 1 & 0 \\ h_2 & 1 \end{matrix}\right) \left(\begin{matrix} 1 & h_3 \\ 0 & 1 \end{matrix}\right) \left(\begin{matrix} 1 & 0 \\ h_4 & 1 \end{matrix}\right)$$ then we get the relations
\begin{equation*}
\begin{aligned}
1&+h_2h_3=1-zw \\
h_1&+h_3+h_1h_2h_3=z^2 \\
h_2&+h_4+h_2h_3h_4=-w^2 \\
1&+h_1h_2+h_1h_4+h_3h_4+h_1h_2h_3h_4=1+zw.
\end{aligned}
\end{equation*}
Rewriting we get 
\begin{equation*}
\begin{aligned}
&h_2h_3=-zw \\
&(1-zw)h_1+h_3=z^2 \\
&h_2+(1-zw)h_4=-w^2 \\
&h_1h_2+(1-zw)h_1h_4+h_3h_4=zw.
\end{aligned}
\end{equation*}
Now on $zw=D$ these relations become 
\begin{equation*}
\begin{aligned}
&h_2h_3=-D \\
&(1-D)h_1+h_3=z^2 \\
&h_2+(1-D)h_4=-w^2 \\
&h_1h_2+(1-D)h_1h_4+h_3h_4=D.
\end{aligned}
\end{equation*}
and assuming for the moment that $D\neq 0$ and $D\neq 1$ we see that 
\begin{equation*}
\begin{aligned}
&h_2=-D/h_3=-zw/h_3 \\
&h_1=(z^2-h_3)/(1-D)= (z^2-h_3)/(1-zw)\\
&h_4=\frac{-h_3w^2+D}{h_3(1-D)}= \frac{-h_3w^2+zw}{h_3(1-zw)}.
\end{aligned}
\end{equation*}
We see that any choice of $h_3\colon \mathbb{C}^2\setminus (\{zw=0\}\cup \{ zw=1\})\to \mathbb{C}^*$ gives a factorization in this part of $\mathbb{C}^2$. In other words the fibre of the fibration
$f^* (\Phi_4)$ over $ \mathbb{C}^2\setminus (\{zw=0\}\cup \{ zw=1\})= \{  (z,w) \in \C^2  : D \in \C \setminus \{0, 1\}$  is $\C^*$ and the fibration is trivial there.
 To get a factorization we must be able to extend this function to the whole of $\mathbb{C}^2$ so that $h_1, h_2$ and $h_4$ still are well-defined.  

When $D=1$ we have
\begin{equation}\label{e:dequal1}
\begin{aligned}
h_3&=z^2 \\
h_2&=-w^2 \\
1&=-w^2h_1+z^2h_4=zw
\end{aligned}
\end{equation} 
and we see that we can pick $h_1$ arbitrary and use the last equation to define $h_4$. In other words the fibre of $f^* (\Phi_4)$ here is $\C$ and again the fibration is trivial when restricted to $D=1$.

Just to complete the picture we remark that over the set $ (\{zw=0\}$, i.e., $D=0$ the fibre  of $f^* (\Phi_4)$  is the cross of axis and the point $(0, 0)$ in the cross of axis is the singular point in those fibres.

A continuous section $s=(h_1,h_2,h_3,h_4)\colon \mathbb{C}^2\setminus \{zw=0 \}$ gives a map $h_3\colon \mathbb{C}^2\setminus \{zw=0\}\to \mathbb{C}^*$ such that $h_3|_{\{zw=1\}}=z^2$ by (\ref{e:dequal1}). Now view $\mathbb{C}^2\setminus \{zw=0\}$ as a bundle over $\mathbb{C}^*$ with fibers $\mathbb{C}^*$ via $zw=D$. Thus $h_3$ gives a family of maps $h_D\colon \mathbb{C}^*\to \mathbb{C}^*$.  Since $h_D$ for $D=1$ is prescribed, the degree of these mappings is $2$ for all $D \in \C\setminus \{0\}$
(depending on some fixed parametrization of the fibre involved, we could as well choose parametrization to get $-2$). Continuous mappings $\C^* \to \C^*$ are homotopic iff they have the same degree. Therefore, if we find a continuous section of the fibration in a neighborhood $U$ of $ D =0$ of the form $U = \{\vert D \vert < \epsilon\}$ having degree $2$ for $D\ne 0$, we can join it to  a section in a neighborhood of $D=1$ (say given by $h_3=z^2$, $h_2=-w/z$, $h_1=0$, and $h_4=w/z$). 
Here is that section:

Define $h_3=w^2/(|w|^{3/2})$ outside $zw=0$ and the other mappings becomes $h_2=-(z|w|^{3/2})/w$, $h_1=(z^2-(w^2/(|w|^{3/2})))/(1-zw)$, and $h_4=-w^2+(z|w|^{3/2}/w)$. These mappings all extends to the whole of $D=0$ and the required relations are satisfied on $zw=0$. Also this extension of $h_3$ gives mapping degree 2. Remark that this section does not avoid the singularity set. Over the point $(0, 0)$ its value is the double point in the cross of axis. Now let's show that there is no continuous section
of  $f^* (\Phi_4)$ avoiding the singularity set $S_4$. Indeed, removing $S_4$ means removing the 
zero point in the fibres over $D=0$, the fibre over $D=0$ becomes now a disjoint union of two copies
of $\C^*$.  Since $D=0$ is a connected set the  section has to be entirely  in one of the copies. 
Now there are two ways of continuing our family of $\C^*$'s parametrized by $D \in C^*$ in into $D=0$:

\begin{eqnarray*}
\left(z, \frac{ D}{z} \right) & \xrightarrow{D \to 0} &  \  (z, 0)
\end{eqnarray*}
 and
 \begin{eqnarray*}
\left(\frac{ D}{z}, z \right) & \xrightarrow{D \to 0} &  \  (0, z)
\end{eqnarray*}

One continuation lands in the $z$-axis, the other in the $w$-axis. Since they are achieved by using different parametrizations of $\C^*$, the corresponding degrees for the map into $\C^*$ are different, $+2$ and $-2$. But shrinking circles  in $D=0$ towards $(0,0)$, one sees that the map to $\C^*$ has to be null-homotopic, i.e., to have degree $0$. 

Next we prove that there is no holomorphic factorization by $4$ factors: The condition $h_2h_3=-zw$ means by division theory in the ring of holomorphic functions that there are 4 possibilities for $h_3$
up to nowhere vanishing functions (units) which are null-homotopic and therefore do not contribute to
degree:  $1$, $z$, $w$ or $zw$. The corresponding degrees are $0$ and $\pm 1$, different from $2$. Thus there is no holomorphic section of $f^* (\Phi_4)$.
Summarizing we have proved:

\begin{Prop}  The matrix $$\left(\begin{matrix} 1+zw & z^2 \\ -w^2 & 1-zw  \end{matrix}\right) \in \mbox{SL}_2( \C [z,w])$$ (which is known to be not factorizable by elementary matrices with polynomial entries) can be factorized as a product of $4$ continuous elementary matrices and as a product of $5$ holomorphic
elementary matrices. Both numbers are minimal in the respective ring. Moreover any factorization of it
by $4$ continuous matrices has to meet the singularity set in the corresponding fibration over $\C^2$.
\end{Prop}

\section{Numerical bounds when $\dim X\le 2$}

We will use obstruction theory to get an upper bound for the number of factors needed when $\dim X\le 2$. 

\subsection{The one-dimensional case}
We begin by describing the situation when $\dim X=1$ and we will show that 4 factors are enough. We write $$\Phi_4(u,z_1,z_2,v)=\left(\begin{matrix} 1 & 0 \\ u & 1 \end{matrix}\right)\left(\begin{matrix} P_1 & P_2 \\ P_3 & P_4 \end{matrix}\right)\left(\begin{matrix} 1 & v \\ 0 & 1 \end{matrix}\right)$$ where $P_1=1+z_1z_2$. The map $\Phi_4$ is submersive outside $\{z_1=z_2=0\}$ which is contained in the set $\tilde{Z}=\Phi_4^{-1}(Z)$ where $$Z=\{\left(\begin{matrix} 1 & b \\ c & d \end{matrix}\right)\}.$$ The stratification of $X$ is $X\supset f^{-1}(Z)\supset \emptyset$. Note that this is not the stratification used to construct the stratified spray. We now construct a section over $f^{-1}(Z)$ using 4 matrices. We simply write $$\left(\begin{matrix} 1 & b \\ c & d \end{matrix}\right)=\left(\begin{matrix} 1 & 0 \\ c & 1 \end{matrix}\right)\left(\begin{matrix} 1 & b \\ 0 & 1 \end{matrix}\right)=\left(\begin{matrix} 1 & 0 \\ c-1 & 1 \end{matrix}\right)\left(\begin{matrix} 1 & 0 \\ 0 & 1 \end{matrix}\right)\left(\begin{matrix} 1 & 0 \\ 1 & 1 \end{matrix}\right)\left(\begin{matrix} 1 & b \\ 0 & 1 \end{matrix}\right).$$ We view this as a section of $f^*\Phi_4$ over $f^{-1}(Z)$ which we can do this because of the constant matrix $\left(\begin{matrix} 1 & 0 \\ 1 & 1 \end{matrix}\right)$ in the factorization. Since $f^*\Phi_4$ is submersive we can extend this section into a neighborhood $U\supset f^{-1}(Z)$. We need to extend this section over the whole of $X$ and this is where we need obstruction theory. The obstructions for extending the section live in the relative cohomology groups $$H^{i+1}(X\setminus f^{-1}(Z),U\setminus f^{-1}(Z),\pi_i(F))$$ for $i\ge 1$ where $F$ is the fiber in the {\it trivial} bundle over $X\setminus f^{-1}(Z)$. The triviality follows since we can pass from fiber $\{z_1z_2=\alpha\}$ to fiber $\{z_1z_2=\beta\}$ via the transformation $T_{\alpha,\beta}(z_1,z_2)=(z_1,\alpha^{-1}\beta z_2,)$ for $\alpha,\beta \neq 0$. We see that the fiber $F=\{z_1z_2=\alpha\}\cong \C^*$.

We calculate the relative cohomology groups $$H^{i+1}(X\setminus f^{-1}(Z),U\setminus f^{-1}(Z),\pi_i(F))$$ for $i\ge 1$. By excision these are the same as $$H^{i+1}(X,f^{-1}(Z),\pi_i(F)).$$ Study the diagram $$ \xymatrix{H^1(f^{-1}(Z),\pi_1(F)) \ar[r] & H^2(X,f^{-1}(Z),\pi_1(F))\ar[r] & H^2(X,\pi_1(F))} $$ Now $X$ is Stein and we may assume that $f^{-1}(Z)$ is a discrete point set. We get $$H^2(X,f^{-1}(Z),\pi_1(F))=0.$$ We also see that $ H^{i+1}(X,f^{-1}(Z),\pi_i(F))=0 $ when $i\ge 2$ in the same way. 

Write $X=\cup_{i=1}^{\infty}X^{i}$ where each $X^i$ is irreducible. Then either $f^{-1}(Z)\cap X^i=X^i$ or $f^{-1}(Z)\cap X^i$ is a point set. On the components where $f^{-1}(Z)\cap X^i=X^i$ we use the explicit factorization we constructed above and these components intersect the rest of the components in a point set. We can therefore assume that $f^{-1}(Z)$ is a point set.

Since all obstructions for extension of the section vanish we get a factorization using 4 elementary matrices with continuous entries. Using the spray we can homotope the section to a holomorphic section and we get a factorization of the matrix using 4 elementary matrices with holomorphic entries. We have

\begin{Thm}\label{dim1}
	Let $X$ be a one-dimensional Stein space and $f\colon X\to \mbox{SL}_2(\mathbb{C})$ be a holomorphic mapping. Then there exists holomorphic mappings $g_1,\dots, g_{4}\colon X\to \mathbb{C}$ such that $$f(x)=\left(\begin{matrix}1 & 0 \\g_1(x) & 1\end{matrix}\right)\left(\begin{matrix}1 & g_2(x) \\ 0 & 1\end{matrix}\right)\left(\begin{matrix}1 & 0 \\g_3(x) & 1\end{matrix}\right)\left(\begin{matrix}1 & g_4(x) \\ 0 & 1\end{matrix}\right).$$
\end{Thm}

\subsection{The two-dimensional case}
We now turn to the case $\dim X=2$. Here we will show that 5 factors are enough. Remember that $$\Phi_5(u,z_1,z_2,z_3,v)=\left(\begin{matrix} 1 & 0 \\ u & 1 \end{matrix}\right)\left(\begin{matrix} P_1 & P_2 \\ P_3 & P_4 \end{matrix}\right)\left(\begin{matrix} 1 & 0 \\ v & 1 \end{matrix}\right)$$ where $P_2=z_1+z_3+z_1z_2z_3$. The map $\Phi_5$ is submersive outside $\{z_1=z_2=z_3=0\}$ which is contained in the set $\tilde{Z}=\Phi_5^{-1}(Z)$ where $$Z=\{\left(\begin{matrix} a & 0 \\ c & a^{-1} \end{matrix}\right)\}.$$The stratification of $X$ is $X\supset f^{-1}(Z)\supset \emptyset$. We now construct a section over $f^{-1}(Z)$ using 4 matrices. We simply write $$\left(\begin{matrix} a & 0 \\ c & a^{-1} \end{matrix}\right)=\left(\begin{matrix} 1 & 0 \\ a^{-1}(c-1) & 1 \end{matrix}\right)\left(\begin{matrix} 1 & a-1 \\ 0 & 1 \end{matrix}\right)\left(\begin{matrix} 1 & 0 \\ 1 & 1 \end{matrix}\right)\left(\begin{matrix} 1 & a^{-1}-1 \\ 0 & 1 \end{matrix}\right).$$ We then add an extra identity matrix at the end and view this as a section of $f^*\Phi_5$ over $f^{-1}(Z)$. We can do this because of the constant matrix $\left(\begin{matrix} 1 & 0 \\ 1 & 1 \end{matrix}\right)$ in our factorization. Since $f^*\Phi_5$ is submersive we can extend this section into a neighborhood $U\supset f^{-1}(Z)$. We need to extend this section over the whole of $X$ and this is where we need obstruction theory. The obstructions for extending the section are located in the relative cohomology groups $$H^{i+1}(X\setminus f^{-1}(Z),U\setminus f^{-1}(Z),\pi_i(F))$$ for $i\ge 1$ where $F$ is the fiber in the {\it trivial} bundle over $X\setminus f^{-1}(Z)$. The triviality follows since we can pass from fiber to fiber via the transformation $T_\alpha(z_1,z_2,z_3)=(\alpha z_1, \alpha^{-1}z_2,\alpha z_3)$ for $\alpha\neq 0$.  The fiber $F$ is given by $F=\{z_1+z_3+z_1z_2z_3=1\}$.

The relative cohomology groups $$H^{i+1}(X\setminus f^{-1}(Z),U\setminus f^{-1}(Z),\pi_i(F))$$ for $i\ge 1$ are easily calculated. By excision these are the same as $$H^{i+1}(X,f^{-1}(Z),\pi_i(F)).$$ First we have $$H^{2}(X,f^{-1}(Z),\pi_1(F))=0$$ trivially since $\pi_1(F)=0$ by Lemma \ref{simplyconnected} below. Study the diagram $$ \xymatrix{H^2(f^{-1}(Z),\pi_2(F)) \ar[r] & H^3(X,f^{-1}(Z),\pi_2(F))\ar[r] & H^3(X,\pi_2(F))} $$ Now both $X$ and $f^{-1}(Z)$ are Stein and we may assume that $f^{-1}(Z)$ has dimension 1 or 0, see below. We get $H^3(X,f^{-1}(Z),\pi_2(F))=0$. We also see that $ H^{i+1}(X,f^{-1}(Z),\pi_i(F))=0 $ when $i\ge 3$ in the same way. 

Write $X=\cup_{i=1}^{\infty}X^{i}$ where each $X^i$ is irreducible. Then either $f^{-1}(Z)\cap X^i=X^i$ or $f^{-1}(Z)\cap X^i$ has strictly lower dimension than $X^i$. On the components where $f^{-1}(Z)\cap X^i=X^i$ we use the explicit factorization we constructed above and these components intersect the rest of the components in one- or zero-dimensional sets. We can therefore assume that $\dim f^{-1}(Z)<2$.

Since all obstructions for extension of the section vanish we get a factorization using 5 elementary matrices with continuous entries. Using the spray we can homotope the section to a holomorphic section and we get a factorization of the matrix using 5 elementary matrices with holomorphic entries. We have

\begin{Thm}\label{dim2}
	Let $X$ be a two-dimensional Stein space and $f\colon X\to \mbox{SL}_2(\mathbb{C})$ be a holomorphic mapping. Then there exists holomorphic mappings $g_1,\dots, g_{5}\colon X\to \mathbb{C}$ such that $$f(x)=\left(\begin{matrix}1 & 0 \\g_1(x) & 1\end{matrix}\right)\left(\begin{matrix}1 & g_2(x) \\ 0 & 1\end{matrix}\right)\left(\begin{matrix}1 & 0 \\g_3(x) & 1\end{matrix}\right)\left(\begin{matrix}1 & g_4(x) \\ 0 & 1\end{matrix}\right)\left(\begin{matrix}1 & 0 \\g_5(x) & 1\end{matrix}\right).$$
\end{Thm}

\begin{Rem}
	Note that any holomorphic map from a two-dimensional Stein space into $\mbox{SL}_2(\mathbb{C})$, which is three-dimensional, is null-homotopic. 
\end{Rem}
\begin{Rem}
	Note that we have not proven that any continuous map $f\colon X\to \mbox{SL}_2(\mathbb{C})$ factors using only 5 matrices. We used the fact that $f^{-1}(Z)$ is Stein in our calculations of $H^{i+1}(X,f^{-1}(Z),\pi_i(F))=0$.
\end{Rem} 

\begin{Lem}  \label{simplyconnected} The fiber $F=\{z_1+z_3+z_1z_2z_3=1\}$ is simply connected.
\end{Lem}

\begin{proof}

Rewrite $z_1+z_3(1+z_1z_2)=1$, put $c=1+z_1z_2$ and we see that part of $F$ is a graph in $ \mathbb{C}^4$ via 
$$(z_1,c)\mapsto (z_1,c,(c-1)/z_1,(1-z_1)/c)$$ 
over $( \mathbb{C}^*)^2$. We have $\pi_1( (\mathbb{C}^*)^2)= \mathbb{Z}^2$ and let us call the generators $(g_1,0)$ and $(0,g_2)$.
 
We need to understand what happens at the points where $z_1=0$ and where $c=0$. If $z_1=0$ then $c=1+z_1z_2=1$, $z_2$ free and $z_3 =1$. So over the point $(0,1)$ we glue a complex line and $(g_1,0)$ becomes contractible. 

Now when $c=0$ then $z_3$ is a free variable, $z_1=1$ and $z_2=-1$ and therefore over the point $(1,0)$ we glue a complex line to get the whole of $F$. Now $(0,g_2)$ becomes contractible and therefore $\pi_1(F)=0$.

\end{proof}

\end{document}